\documentclass[12pt]{article}   	
\usepackage{geometry}                		
\usepackage[parfill]{parskip}    		
\usepackage{graphicx}				
\usepackage[hyperindex=true,pagebackref,colorlinks=true,pdfpagemode=none,urlcolor=blue,linkcolor=blue,citecolor=blue,pdfstartview=FitH]{hyperref}
\usepackage{amsmath,amsfonts}
\usepackage{color}
\usepackage{amsthm}
\usepackage{boolexpr}
\usepackage{amssymb,latexsym,pstricks,pst-plot}
\usepackage[english]{babel}
\usepackage{pgf}
\usepackage{tikz}
\usetikzlibrary{arrows,automata}
\usepackage[latin1]{inputenc}
\usepackage[UKenglish]{isodate}
\allowdisplaybreaks

\usepackage{amssymb}
\newtheorem{theorem}{Theorem}[section]
\newtheorem{lemma}[theorem]{Lemma}

\newtheorem{corollary}[theorem]{Corollary}

\newtheorem{conjecture}[theorem]{Conjecture}

\title{Extremely Symmetric primes}
\author{Rob Burns}

\begin{document}
\maketitle
\begin{abstract}
We introduce extremely symmetric primes and provide some elementary properties of these.
\end{abstract}

\section{Introduction}
\label{intro}

Symmetric primes were introduced by Fletcher, Lindgren and Pomerance in \cite{FLETCHER199689}. Let $p$ and $q$ be primes and consider the rectangle determined by the positive $x$ and $y$ axes and the point $(p/2, q/2)$. The main diagonal is the line from the origin to $(p/2, q/2)$. The primes $p$ and $q$ are defined to be symmetric if the number of lattice points above and below the main diagonal of this rectangle are equal. An equivalent and more useful characterisation, established in \cite{FLETCHER199689}, is that $p$ and $q$ are symmetric if $(p - 1, q - 1) = | p - q |$, where $(  \, \, , \,  )$ represents the greatest common divisor. Twin primes are therefore symmetric. It follows that, if $(p, q)$ is a symmetric pair with $p < q$, then $q \leq 2p - 1$ and, if $p$ and $2p - 1$ are prime, then $(p, 2p - 1)$ is a symmetric pair. If a prime $p$ is a member of no symmetric pair then it is called asymmetric.

The density of the symmetric primes has been studied in \cite{FLETCHER199689} and \cite{bpp2019}. Define the constant $\eta$ by
$$
\eta := 1 - \frac{1 + \log \log {}2}{\log {}2}
$$
and let $S(x)$ denote the number of symmetric primes less than $x$. Then, from \cite{bpp2019},
$$
S(x) \leq \frac{\pi(x)}{(\log {}x)^{\eta}} (\log \log {}x)^{\mathcal{O}(1)}
$$
when $x$ is large enough and it is conjectured that
$$
S(x) = \mathcal{O} \left( \frac{\pi(x)}{(\log {}x)^{\eta + o(1)}} \right ).
$$
The same paper showed that, for any integer $m$, there is a sequence of $m$ consecutive primes such that any two primes in the sequence forms a symmetric pair. One consequence of this is that the number of symmetric primes is infinite. The upper bound on $S(x)$ can be used to show that the sum of the reciprocals of the symmetric primes is finite. If $\mathbb{P}$ is the set of primes, $\sum_{p \in \mathbb{P}} \frac{1}{p}$ diverges. Therefore, the number of asymmetric primes is infinite.

In this paper we will examine symmetric sequences which are finite sequences of primes in which any two primes in the sequence is a symmetric pair. We note that there are no infinite symmetric sequences of primes because, if $p$ is the smallest prime in a symmetric sequence, all other primes in the sequence must be $\leq 2p - 1$.

\bigskip	

\section{Symmetric triples}
\label{extreme}
A symmetric triple is a set of three primes, any two of which forms a symmetric pair. We write the set as $(p, q, r)$ where $p < q < r \leq 2p - 1$.
\bigskip

\begin{lemma}
There are no symmetric triples in the form $(p, p + 2, p + 6)$. The set $(p, p + 4, p + 6)$ is a symmetric triple only if $p = 1 \pmod {12}$.
\end{lemma} 
\begin{proof}
If $p = 1 \pmod{3}$, then $p + 2$ is not prime. If $p = 2 \pmod{3}$, then $\{p, p +6 \}$ cannot be a symmetric pair as $6 \nmid p - 1$. This proves the first statement.
If $\{ p, p + 4  \}$ is a symmetric pair, $p = 1 \pmod 4$. If $\{ p, p + 6  \}$ is a symmetric pair, $p = 1 \pmod 6$. The second statement follows.
\end{proof}
Examples of such symmetric triples include $\{ 13, 17, 19 \}$, $\{ 37, 41, 43 \}$, $\{97, 101, 103 \}$ etc. On the other hand, $\{7, 11, 13 \}$ is not a symmetric triple as $(7 - 1, 11 - 1 ) \neq 4$.

\bigskip

Another example of a symmetric triple is the set $(p, p + \frac{p-1}{2^{k+1}}, p + \frac{p-1}{2^k})$ for $k \geq 0$ provided:
\begin{itemize}
	\item $2^{k+1}$ divides $p - 1$ and
	\item all three numbers are prime.
\end{itemize}

\bigskip

\begin{lemma}
There are no symmetric pairs $(p, q)$ with $\frac{3p - 1}{2} < q < 2p-1$ . 
\end{lemma} 
\begin{proof}
If $\frac{3p - 1}{2} < q < 2p-1$, then $\frac{p - 1}{2} < q - p < p-1$. However, $p-1$ has no divisors in that range.
\end{proof}

\bigskip

\begin{corollary}
If $(p, q, r)$ is a symmetric triple then  $q < \frac{3p-1}{2}$ and either $r = 2p - 1$ or $q < r \leq \frac{3p-1}{2}$.
\end{corollary}

\bigskip
We will characterise symmetric triples in which $r = 2p - 1$ in the next section.

\section{Extreme symmetric sequences}
\label{extreme}
We will call a symmetric sequence extreme if the smallest prime in the sequence is $p$ and the largest is $2p - 1$. In such a sequence the gap between the smallest and the largest prime is as wide as possible. We will call a prime extremely symmetric if it is the smallest prime in at least one extreme symmetric sequence.

\bigskip

\begin{theorem}
\label{extremetriple}
Suppose $\{ p, q, 2p - 1 \}$ is an extreme symmetric triple and write $q = p + d$ where $d | p-1$. Then $d$ equals either $\frac{p-1}{3}$ or $\frac{p-1}{2}$.
\end{theorem}
\begin{proof}
Since $d | p-1$ and $d \neq p-1$, $d \leq \frac{p-1}{2}$. We also have that 
$$
(2p - 1) - (p + d) = p - d - 1 | p + d - 1.
$$
Since $p - d - 1 \neq p + d - 1$,  $p - d - 1 \leq \frac{p + d - 1}{2}$. Therefore, $d \geq \frac{p - 1}{3}$ and
$$
\frac{p - 1}{3} \leq d \leq \frac{p - 1}{2}.
$$
Since $d | p-1$, it must equal either $\frac{p - 1}{3}$ or $\frac{p - 1}{2}$.
\end{proof}

\bigskip

\begin{corollary}
For each prime $p$ there are at most two possible extreme symmetric triples beginning with $p$ and at most one extreme symmetric sequence of length $4$ beginning with $p$. There are no extreme symmetric sequences of length greater than $4$. 
\end{corollary}

\section{Density}
\label{density}
In this section we examine the density of extreme symmetric primes. If $A$ is a set, we denote the number of elements in $A$ by $\# A$. We first look at triples $(p, \frac{3p - 1}{2}, 2p - 1)$. Let $S(x)$ denote the number of primes $p \leq x$ such that $(p, \frac{3p - 1}{2}, 2p - 1)$ is an extreme symmetric triple. First, notice that if $p$, $\frac{3p - 1}{2}$ and $2p - 1$ are all prime, then this is an extreme symmetric triple as the gap between each prime is $\frac{p-1}{2}$ which divides $p-1$ and $\frac{p-1}{2}$. Since $p$ is odd, $\frac{3p - 1}{2}$ is integer. Writing $p = 2k + 1$ for $k \geq 1$, we have
$$
(p, \frac{3p - 1}{2}, 2p - 1) = (2k + 1, 3k + 1, 4k + 1).
$$
Therefore,
$$
S(x) = \# \Big\{ k \leq \frac{x}{2}: 2k + 1, 3k + 1, 4k + 1 \in \mathbb{P} \Big\}.
$$

It is not known whether there are infinitely many primes $p$ such that $2p - 1$ is also prime. Hence, it is also unknown whether $S(x)$ is finite or infinite. Let $F = \{ f_1, . . . , f_m \}$ be a set of linear functions in one variable with integer coefficients and positive leading coefficients. F is called admissible if $\, \prod_{i = 1}^{m} f_i \,$ has no fixed prime divisor, meaning that for any prime $p$, there is some integer $k_p$, such that $p$ does not divide $\prod_{i = 1}^{m} f_i(k_p)$. The three linear functions $f_1(k) =  2k + 1$, $f_2(k) = 3k + 1$, $f_3(k) = 4k + 1$ form an admissible set. The following well know conjecture, known as the Prime $k$-tuples Conjecture, would establish that $S(x)$ is infinite.

\bigskip

\begin{conjecture} Let $F = \{f_1, . . . , f_m \}$ be an admissible set of linear functions. Then there are infinitely many integers $k$ such that all of $f_1(k), . . . , f_m(k)$ are primes.
\end{conjecture}

\bigskip

We use Cramer's random model of the primes to obtain a heuristic estimate for $S(x)$. The paper \cite{Caldwell:aa} provides a good summary of the method and numerous examples of its use. Under Cramer's model, the probability of an integer $k$ being prime is about $\frac{1}{\log k}$. Let $\{f_1, . . . , f_m \}$ be a set of linear functions. If the probability that $f_i(k)$ is prime is independent of the probability of $f_j(k)$ being prime for all $1 \leq i, j \leq m$, then the probability that all $f_i(k)$ are prime should be approximately
$$
\left( \prod_{i=1}^{m} \log (f_i(k)) \right)^{-1}.
$$ 
By integrating this function we obtain an estimate for the number of $k$ such that all $\{f_i(k)\}$ are prime. Cramer's method adjusts this integral by a factor which measures how close the  $\{f_i\}$ are to being independent. The resulting estimate is:
$$
\prod_{p \in \mathbb{P}} \frac{(1 - \omega(p)/p)}{\left(1 - 1/p \right)^m} \int_2^N \frac{dy}{\prod_{i=1}^{m} \log (f_i(y))} \, \, .
$$
where $\omega(p)$ is defined by
$$
\omega(p) := \# \Big\{k: 0 \leq k < p: \prod_{i=1}^{m} f_i(k) = 0 \pmod{p} \Big\}.
$$

\bigskip

When the set of functions is $\{ 2k + 1, 3k + 1, 4k + 1 \}$ we have $\omega(2) = 1, \omega(3) = 2$ and $\omega(p) = 3$ for $p \geq 5$. We then get the following estimate for $S(x)$:
\begin{equation}
\label{est2k+1}
S(x) \sim \frac{9}{2} \prod_{p \geq 5} \frac{(1 - 3/p)}{\left(1 - 1/p \right)^3} \int_1^{x/2} \frac{dy}{ \log (2y) \log(3y) \log(4y) } \, \, .
\end{equation}
The constant $\prod_{p \geq 5} \frac{(1 - 3/p)}{\left(1 - 1/p \right)^3}$ is approximately $0.635166354604222$. Table \ref{table2k+1} shows a comparison of actual values and estimated values of $S(x)$. Earlier we noted that the initial prime $p$ in the triple was of the form $2k+1$. In fact, it can be seen that $p$ must be of the form $12k + 1$, since if $p = 5 \pmod{12}$ then $3$ divides $2p-1$ and if $p \in \{7, 11 \} \pmod{12}$ then one of $\frac{3p-1}{2}$ or $2p-1$ is even. So instead of counting prime triples of the form $( 2k + 1, 3k + 1, 4k + 1 )$, we could consider primes $( 12k + 1, 18k + 1, 24k + 1 )$. However, this produces the same estimate for $S(x)$.

\bigskip

\begin{table}[tbp]
  \centering
   \begin{tabular}{ | c | c | p{4cm} |}
 \hline
  & Actual & Estimated\\
 x & number & number \\  \hline
1,000 & 1 & 8  \\
10,000 & 15 & 27 \\ 
100,000 & 111 & 122 \\ 
1,000,000 & 623 & 659 \\
10,000,000 & 3990 & 3988 \\
100,000,000 & 26179 & 26041 \\ \hline
  \end{tabular}
  \caption{Table of actual and estimated values of triples $(p, \frac{3p - 1}{2}, 2p - 1)$}
  \label{table2k+1}
\end{table}

\bigskip

We now consider extreme symmetric triples of the form $(p, \frac{4p - 1}{3}, 2p - 1)$. Let $T(x)$ denote the number of primes $p \leq x$ such that $(p, \frac{4p - 1}{3}, 2p - 1)$ is an extreme symmetric triple. In order for $\frac{4p - 1}{3}$ to be an integer, we must have $p = 1 \pmod{3}$. If $p$, $\frac{4p - 1}{3}$ and $2p - 1$ are all prime, then this is an extreme symmetric triple as the two gaps are $\frac{p-1}{3}$ which divides $p-1$ and $\frac{2(p-1)}{3}$ which divides $\frac{4p-1}{3} - 1$ when $p = 1 \pmod{3}$. Writing $p = 3k + 1$, where $k \geq 1$, we have
$$
(p, \frac{4p - 1}{3}, 2p - 1) = (3k + 1, 4k + 1, 6k + 1).
$$
Therefore,
$$
T(x) = \# \{ k \leq \frac{x}{3}: 3k + 1, 4k + 1, 6k + 1 \in \mathbb{P} \}.
$$
For the set of functions $\{ 3k + 1, 4k + 1, 6k + 1 \}$ we have $\omega(2) = 1, \omega(3) = 1$ and $\omega(p) = 3$ for $p \geq 5$. We then get the following estimate for $T(x)$:
\begin{equation}
\label{est3k+1}
T(x) \sim 9 \prod_{p \geq 5} \frac{(1 - 3/p)}{\left(1 - 1/p \right)^3} \int_1^{x/3} \frac{dy}{ \log (3y) \log(4y) \log(6y) } \, \, .
\end{equation}

\bigskip

\begin{table}[tbp]
  \centering
   \begin{tabular}{ | c | c | p{4cm} |}
 \hline
  & Actual & Estimated\\
 x & number & number \\  \hline
1,000 & 9 & 11  \\
10,000 & 26 & 37 \\ 
100,000 & 142 & 165 \\ 
1,000,000 & 864 & 887 \\
10,000,000 & 5326 & 5359 \\
100,000,000 & 34863 & 34957 \\ \hline
  \end{tabular}
  \caption{Table of actual and estimated values of triples $(p, \frac{4p - 1}{3}, 2p - 1)$}
  \label{table3k+1}
\end{table}

\bigskip

Table \ref{table3k+1} shows a comparison of actual values and estimated values of $T(x)$. The estimates for $S(x)$ and $T(x)$ indicate that there are roughly $\frac{4}{3}$ more extreme symmetric triples of the form $(p, \frac{4p - 1}{3}, 2p - 1)$ than of the form $(p, \frac{3p - 1}{2}, 2p - 1)$.

\bigskip

Finally, we look at the extreme symmetric quadruple $(p, \frac{4p - 1}{3}, \frac{3p - 1}{2}, 2p - 1)$. This sequence forms a complete sub-graph of the graph, described in \cite{bpp2019}, in which nodes are labelled by primes and paths exist between two primes if they form a symmetric pair. The first two extreme symmetric sequences of this form are $(661, 881, 991, 1321)$ and $(6121, 8161, 9181, 12241)$. In order for both $\frac{4p - 1}{3}$ and $\frac{3p - 1}{2}$ to be integers we must have $p = 1 \pmod{6}$. Writing $p$ as $6k+1$ and letting $W(x)$ denote the number of extreme symmetric quadruples with smallest prime $\leq x$, we have
$$
W(x) = \# \{ k \leq \frac{x}{6}: 6k + 1, 8k + 1, 9k + 1, 12k + 1 \in \mathbb{P} \}.
$$
For the set of functions $\{ 6k + 1, 8k + 1, 9k + 1, 12k + 1  \}$ we have $\omega(2) = 1, \omega(3) = 1$ and $\omega(p) = 4$ for $p \geq 5$. We then get the following estimate for $W(x)$:
\begin{equation}
\label{est6k+1}
W(x) \sim 27 \prod_{p \geq 5} \frac{(1 - 4/p)}{\left(1 - 1/p \right)^4} \int_1^{x/6} \frac{dy}{ \log (6y) \log(8y) \log(9y) \log(12y) }  \, \, .
\end{equation}

\bigskip

The product $\prod_{p \geq 5} \frac{(1 - 4/p)}{\left(1 - 1/p \right)^4}$ is approximately $0.3074948895$. Table \ref{tablequad} shows a comparison of actual values and estimated values of $W(x)$.

\bigskip

\begin{table}[tbp]
  \centering
   \begin{tabular}{ | c | c | p{4cm} |}
 \hline
  & Actual & Estimated\\
 x & number & number \\  \hline
1,000 & 1 & 1  \\
10,000 & 2 & 3 \\ 
100,000 & 9 & 12 \\ 
1,000,000 & 43 & 51 \\
10,000,000 & 249 & 258 \\
100,000,000 & 1465 & 1452 \\ \hline
  \end{tabular}
  \caption{Table of actual and estimated values of quadruples $(p, \frac{4p - 1}{3}, \frac{3p - 1}{2}, 2p - 1)$}
  \label{tablequad}
\end{table}

\bigskip

Sets of the form $\{ 2k + 1, 3k + 1, 4k + 1 \}$ and $\{ 3k + 1, 4k + 1, 6k + 1 \}$ can be sieved using Brun's method (see, for example, \cite{Nathanson_1996}) to give an upper bound:
\begin{equation}
\label{brun3}
S(x), T(x) =  \mathcal{O} \left( x \left( \frac{\log \log x}{\log x} \right)^3 \, \right).
\end{equation}

\bigskip

Similarly, the set $\{ 6k + 1, 8k + 1, 9k + 1, 12k + 1 \}$ can be sieved to give an upper bound:
\begin{equation}
\label{brun4}
W(x) =  \mathcal{O} \left( x \left( \frac{\log \log x}{\log x} \right)^4 \, \right).
\end{equation}

\bigskip

The estimates (\ref{brun3}) and  (\ref{brun4}) are consistent with the upper bounds for $S$, $T$ and $W$ in (\ref{est2k+1}), (\ref{est3k+1}) and (\ref{est6k+1}). 

\bigskip

\bibliographystyle{plain}
\begin{small}
\bibliography{SymmetricPrimes}

\begin{thebibliography}{1}

\bibitem{bpp2019}
William Banks, Paul Pollack, and Carl Pomerance.
\newblock Symmetric primes revisited.
\newblock {\em Integers}, 19, 2019.

\bibitem{Caldwell:aa}
Chris~K. Caldwell.
\newblock An amazing prime heuristic.

\bibitem{FLETCHER199689}
Peter Fletcher, William Lindgren, and Carl Pomerance.
\newblock Symmetric and asymmetric primes.
\newblock {\em Journal of Number Theory}, 58(1):89 -- 99, 1996.

\bibitem{Nathanson_1996}
Melvyn~B. Nathanson.
\newblock Additive number theory.
\newblock {\em Graduate Texts in Mathematics}, 1996.

\end{thebibliography}
\end{small}

\end{document}